\documentclass[12pt]{amsart}
\usepackage[latin1]{inputenc}
\usepackage{times}
\usepackage{amsfonts}
\usepackage{amssymb}
\usepackage{bbm}
\usepackage{times}
\usepackage{amssymb}
\usepackage{amscd}
\usepackage{graphicx}

\usepackage{amsmath}
\usepackage{amssymb}

\usepackage{amsmath}
\usepackage{amsfonts}
\usepackage{amscd}
\usepackage{latexsym}
\usepackage{amsthm}
\usepackage{amssymb}
\usepackage{amsmath}
\usepackage{setspace}
\theoremstyle{plain}
\newtheorem{theorem}{Theorem}[section]
\newtheorem{corollary}[theorem]{Corollary}
\newtheorem{lemma}[theorem]{Lemma}

\newtheorem{definition}[theorem]{Definition}

\def\ep{\varepsilon}

\newcommand{\sg}{\sigma}

\begin{document}
\title[tracially nuclear dimensional]{{\bf certain  tracially nuclear dimensional for certain crossed product ${\rm C^*}$-algebras}}
\author{Qingzhai Fan and  Jiahui Wang}

\address{Qingzhai Fan\\ Department of Mathematics\\  Shanghai Maritime University\\
Shanghai\\China
\\  201306 }

\address{Jiahui  Wang\\ Department of Mathematics\\ Shanghai Maritime University\\
Shanghai\\China
\\  201306 }
\email{641183019@qq.com}


\thanks{2000 \emph{Mathematics Subject Classification\rm{:}} 46L35, 46L05, 46L80}
\begin{abstract} Let $\Omega$ be a class of unital ${\rm C^*}$-algebras which have the second type tracial nuclear dimensional at moat $n$ (or have  tracial nuclear dimensional at most $n$). Let $A$ be an infinite dimensional  unital simple  ${\rm C^*}$-algebra such that  $A$ is asymptotical tracially in $\Omega$. Then ${\rm T^2dim_{nuc}}(A)\leq n$  (or ${\rm Tdim_{nuc}}(A)\leq n$).
 As an application, let $A$ be an infinite dimensional  simple separable amenable unital
 ${\rm C^*}$-algebra with ${\rm T^2dim_{nuc}}(A)\leq n$ (or ${\rm Tdim_{nuc}}(A)\leq n$).
Suppose that  $\alpha:G\to {\rm Aut}(A)$ is  an action of a finite group $G$ on $A$
 which has the  tracial Rokhlin property.  Then ${\rm T^2dim_{nuc}}({{\rm C^*}(G, A,\alpha)})\leq n$ (or ${\rm Tdim_{nuc}}$ $({{\rm C^*}(G, A,\alpha)})\leq n$).
 \end{abstract}
\maketitle

\section{Introduction}
The Elliott program for the classification of amenable
 ${\rm C^*}$-algebras might be said to have begun with the ${\rm K}$-theoretical
 classification of AF algebras in \cite{E1}. Since then, many classes of
 ${\rm C^*}$-algebras have been classified by the Elliott invariant.
    A major next step was the classification of simple
 $\rm AH$ algebras without dimension growth (in the real rank zero case see \cite{E6}, and in the general case
 see \cite{GL}).
 A crucial intermediate step was Lin's axiomatization of  Elliott-Gong's decomposition theorem for  simple $\rm AH$ algebras of real rank zero (classified by Elliott-Gong in \cite{E6}) and Gong's decomposition theorem (\cite{G1}) for simple $\rm AH$ algebras (classified by Elliott-Gong-Li in \cite{GL}).  Lin introduced the concepts of   $\rm TAF$ and $\rm TAI$ (\cite{L0} and \cite{L1}). Instead of assuming inductive limit structure, Lin started with a certain abstract (tracial) approximation property.  This led eventually to the classification of  simple separable amenable stably finite  ${\rm C^*}$-algebras with finite nuclear dimension in the UCT class (see \cite{GLN1}, \cite{GLN2}, \cite{EZ5}, \cite{TWW1}).

Inspired by    Lin's  tracial interval algebras in \cite{L1},  Elliott and  Niu in \cite{EZ} considered  the natural   notion of
 tracial approximation by  other classes of ${\rm C^*}$-algebras. Let $\Omega$ be a class of unital ${\rm C^*}$-algebras. Then the class
 of simple  separable  ${\rm C^*}$-algebras which can be tracially
approximated by ${\rm C^*}$-algebras in $\Omega,$ denoted by  ${\rm TA}\Omega$,  is
defined as follows.  A  simple unital ${\rm C^*}$-algebra $A$ is said to
      belong to the class ${\rm TA}\Omega$ if,  for any
 $\ep>0,$ any finite
subset $F\subseteq A,$ and any non-zero element $a\geq 0,$  there
are a  projection $p\in A$ and a ${\rm C^*}$-subalgebra $B$ of $A$ with
$1_B=p$ and $B\in \Omega$ such that

$(1)$ $\|xp-px\|<\ep$, for all $x\in  F$,

$(2)$ $pxp\in_\ep B$, for
all $x\in  F$, and

 $(3)$ $1-p$ is Murray-von Neumann equivalent to a
projection in $\overline{aAa}$.

The Rokhlin property in ergodic theory was adapted to
the context
 of von Neumann algebras by  Connes in \cite{AC}. It was adapted by  Herman
 and  Ocneanu  for  automorphisms of  UHF algebras in \cite{HHO1}.  Izumi \cite{Izu},  R{\o}rdam \cite{R11} and
 Kishimoto \cite{K3}
 considered the Rokhlin property in  a much more general ${\rm C^*}$-algebra context. More recently,  Phillips  and  Osaka studied actions of a
 finite group and  of the  group $\mathbb{Z}$ of integers   on certain  simple ${\rm C^*}$-algebras which have a modified  Rokhlin property
  in \cite{OP} and \cite{P2}.

In \cite{P2}, Phillips show that
   let $A$ be an infinite-dimensional simple separable unital
${\rm C^*}$-algebra with tracial topological  rank zero. Let $\alpha:G\to {\rm Aut}(A)$ is an action of a finite group $G$ on $A$ which has the tracial Rokhlin property. Then  ${\rm C^*}( G, A,\alpha)$ has tracial topological rank zero.

Inspired by Phillips's above result, in  \cite{Q9}, Fan and Fang  showed  the following result:

  Let $\Omega$ be a class of  unital
${\rm C^*}$-algebras such that $\Omega$ is closed under passing to  unital hereditary ${\rm C^*}$-subalgebra and tensoring matrix algebras. Let $A\in$ $\rm {TA}\Omega$ (tracial approximated by ${\rm C^*}$-algebras in $\Omega$) be  an infinite-dimensional  simple  unital ${\rm C^*}$-algebra.
 Suppose that $\alpha:G\to {\rm{Aut}}(A)$ is an action of a finite group $G$ on $A$
 which has  tracial Rokhlin property. Then the crossed product ${\rm C^*}$-algebra
 ${\rm C^*}( G, A,\alpha)$ belongs to $\rm{TA}\Omega$.

Note that:  a similar  result was  obtained independently using different method by H. H. Lee and H. Osaka in \cite {LO}.

 Let  $\Omega$ be a class of  unital
${\rm C^*}$-algebras such that  for any $B\in \Omega$, $B$  has  property $H$ (we assume that  property $H$ is closed under tensoring with matrix algebras and
 under passing to  unital
hereditary ${\rm C^*}$-subalgebras). If we show that $A$ has the property $H$ for any unital simple separable $A\in \rm TA\Omega$, then by the above result, we can get the following result:

\ Let $A$ be a unital
${\rm C^*}$-algebras such that $A$ has property $H$.
 Suppose that $\alpha:G\to$ ${\rm {Aut}}(A)$ is an action of a finite group $G$ on $A$
 which has the  tracial Rokhlin property. Then the crossed product ${\rm C^*}$-algebra
 ${\rm C^*}( G, A,\alpha)$ also has  property $H$.

 In the quest to classify simple separable nuclear  ${\rm C^*}$-algebras, as suggested by G. A. Elliott, it has become necessary to invoke some regularity
 property of the  ${\rm C^*}$-algebras. There are three regularity properties of particular interest: tensorial absorption of the Jiang-Su algebra $\mathcal{Z}$, also called $\mathcal{Z}$-stability;  finite nuclear dimension; and strict comparison of positive elements.
 The latter can be reformulated as an algebraic property of Cuntz semigroup, called almost unperforated. Winter and Toms have conjectured that these three fundamental properties are equivalent for all separable, simple, nuclear  ${\rm C^*}$-algebras.

Hirshberg and Oroviz introduced tracially $\mathcal{Z}$-stable in \cite{HO}, they  showed that $\mathcal{Z}$-stable is equivalent to tracially $\mathcal{Z}$-stable for unital simple separable amenable
  ${\rm C^*}$-algebra in \cite{HO}.

Inspired by  Hirshberg and Oroviz's tracially $\mathcal{Z}$-stable, and  inspired by the work of Elliott,  Gong,  Lin, and  Niu in \cite {EGLN2}, \cite{GL2} and \cite{GL3},  and also in order to search a tracial version of Toms-Winter conjecture, Fu and Lin  introduced  some class of tracial nuclear dimensional ${\rm C^*}$-algebras in \cite{FU} and \cite{FL}.
Let $\Omega$ be a class of ${\rm C^*}$-algebras. Fu and Lin introduced  a class of ${\rm C^*}$-algebras which was called asymptotical tracially in $\Omega$, and denoted this class of ${\rm C^*}$-algebras by $\rm AT\Omega$ (see Definition \ref{def:2.7}).

In this paper, we will show the following results:

Let $\Omega$ be a class of  unital
separable  $\rm {C^*}$-algebras and let $A$ be a unital  simple separable infinite dimensional  amenable $\rm {C^*}$-algebra.  Then  $A\in \rm {AT}\Omega$ if and only if $A\in \rm{TA(AT\Omega)}$ (see Definition \ref{def:2.7} and Definition \ref{def:2.2}).

Let $\Omega$ be a class of unital ${\rm C^*}$-algebras which have the second type  tracial nuclear dimensional at moat $n$ (or have tracial nuclear dimensional at most $n$). Let $A$ be a unital simple ${\rm C^*}$-algebra  such  that  $A$ is asymptotical tracially in $\Omega$. Then ${\rm T^2dim_{nuc}}(A)\leq n$ (or ${\rm Tdim_{nuc}}(A)\leq n$) (see Definition \ref{def:2.9} and Definition \ref{def:2.7}).
	
As applications, we can get the following results:

 Let $A$ be an infinite dimensional  simple separable amenable unital
 ${\rm C^*}$-algebra such that ${\rm T^2dim_{nuc}}(A)\leq n$.
Suppose that  $\alpha:G\to {\rm Aut}(A)$ is  an action of a finite group $G$ on $A$
 which has the  tracial Rokhlin property.  Then ${\rm T^2dim_{nuc}}({{\rm C^*}(G, A,\alpha)})\leq n$.

 Let $A$ be an infinite dimensional  simple separable amenable unital
 ${\rm C^*}$-algebra such that ${\rm Tdim_{nuc}}(A)\leq n$.
Suppose that  $\alpha:G\to {\rm Aut}(A)$ is  an action of a finite group $G$ on $A$
 which has the  tracial Rokhlin property.  Then ${\rm Tdim_{nuc}}({{\rm C^*}(G, A,\alpha)})\leq n$ (\textbf{Note that}: this  result also  can be obtained by Theorem 9.3 of  \cite{FL} and Theorem 3.1 of \cite{EFF} and Theorem \ref{thm:3.3}).

	\section{Preliminaries and definitions}

Let  ${\rm M}_{n}(A)_+$ denote
the positive elements of  ${\rm M}_{n}(A)$.  Given $a,~ b\in {\rm M}_{n}(A)_+,$
we say that $a$ is Cuntz subequivalent to $b$ (written $a\precsim b$) if there is a sequence $(v_n)_{n=1}^\infty$
of elements of ${\rm M}_{n}(A)$ such that $$\lim_{n\to \infty}\|v_nbv_n^*-a\|=0.$$
We say that $a$ and $b$ are Cuntz equivalent (written $a\sim b$) if $a\precsim b$ and $b\precsim a$. We write $\langle a\rangle$ for the equivalence class of $a$.

The object ${\rm W}(A):={\rm M}_{\infty}(A)_+/\sim$
 will be called the Cuntz semigroup of $A$.  ${\rm W}(A)$ becomes  an ordered  semigroup   when equipped with the addition operation
$$\langle a\rangle+\langle b\rangle=\langle a \oplus b\rangle,$$
 and the order relation
$$\langle a\rangle\leq \langle b\rangle\Leftrightarrow a\precsim b.$$

Let $A$ be a  unital ${\rm C^*}$-algebra. Recall that a positive element $a\in A$ is called purely positive if $a$ is
not Cuntz equivalent to a projection. Let $A$ be a stably finite ${\rm C^*}$-algebra and let  $a\in A$ be  a positive element. Then $a$ is a purely positive element or $a$ is Cuntz equivalent to a projection.

Given $a$ in $A_+$ and $\varepsilon>0,$ we denote by $(a-\varepsilon)_+$ the element of ${\rm C^*}(a)$ corresponding (via the functional calculus) to the function $f(t)={\max (0, t-\varepsilon)},$ $t\in \sigma(a)$. By the functional calculus, it follows in a straightforward manner that $((a-\varepsilon_1)_+-\varepsilon_2)_+=(a-(\varepsilon_1+\varepsilon_2))_+.$

The following facts are  well known.
\begin{theorem}{\rm (\cite{PPT}, \cite{HO}, \cite{P3}, \cite{RW}.)} \label{thm:2.1} Let $A$ be a ${\rm C^*}$-algebra.

 $(1)$ Let $a,~ b\in A_+$ and any  $\varepsilon>0$  be such that
$\|a-b\|<\varepsilon$.  Then there is a contraction $d$ in $A$ with $(a-\varepsilon)_+=dbd^*$.

 $(2)$ Let $a$ be a  positive element  of $A$
not Cuntz equivalent to a projection. Let  $\delta>0$,  and let $f\in C_0(0,1]$ be a non-negative function with $f=0$ on $(\delta,1),$  $f>0$ on $(0,\delta)$,
and $\|f\|=1$.  Then $f(a)\neq 0$
and  $(a-\delta)_++f(a)\precsim a.$

$(3)$ Let $a,~ b\in A$ satisfy $0\leq a\leq b$. Let $\varepsilon\geq 0$.
Then $(a-\varepsilon)_+\precsim(b-\varepsilon)_+$ (Lemma 1.7 of \cite{P3}).
\end{theorem}

Inspired by  Lin's  tracial approximation by
  interval algebras in \cite{L1},  Elliott and  Niu in \cite{EZ} considered  the natural   notion of
 tracial approximation by  other classes of ${\rm C^*}$-algebras.

 Let $\Omega$ be a class of unital ${\rm C^*}$-algebras. Then the class
 of simple  separable  ${\rm C^*}$-algebras which can be tracially
approximated by ${\rm C^*}$-algebras in $\Omega,$ denoted by  ${\rm TA}\Omega$,  is defined as follows:

\begin{definition}\label{def:2.2}{\rm (\cite{EZ},~\cite{L1}.)}
A  simple unital ${\rm C^*}$-algebra $A$ is said to
    belong to the class ${\rm TA}\Omega$ if,  for any
 $\varepsilon>0,$ any finite
subset $\mathcal{G}\subseteq A,$ and any non-zero element $a\geq 0,$  there
are a  projection $p\in A$ and a ${\rm C^*}$-subalgebra $B$ of $A$ with
$1_B=p$ and $B\in\Omega$ such that

$(1)$ $\|xp-px\|<\varepsilon$, for all $x\in  \mathcal{G}$,

$(2)$ $pxp\in_\varepsilon B$, for
all $x\in  \mathcal{G}$, and

 $(3)$ $1-p\precsim a$.
\end{definition}

 Let $a$ and $b$ be two positive elements in a  $\rm C^*$-algebra  $A$. We
write $[a]\leq[b]$ (cf Definition 3.5.2 in \cite{L2}), if there exists a partial isometry $v\in A^{**}$
such that, for every $c\in {\rm Her}(a),~ v^{*}c, ~ cv\in A,~ vv^*=P_{a}$, where $P_{a}$ is the range projection of $a$ in $A^{**}$,
and $v^*cv\in {\rm Her}(b)$. We write $[a]=[b]$ if $v^*{\rm Her}(a)v={\rm Her}(b)$.
Let $n$ be a positive integer. We write $n[a]\leq [b] $, if there
are $n$ mutually orthogonal positive elements $b_1,~ b_2,~ \cdots,~
b_n\in {\rm Her}(b)$ such that $[a]\leq[b_i],$ for  $i=1, ~ 2,~ \cdots,~ n$. (cf. Definition 1.1 in \cite{PZ}, Definition 3.2 in \cite{OMT},  or Definition 3.5.2 in \cite{L2}.)

Let $0<\sg_1<\sg_2\leq 1$ be two positive numbers. Define
$$
f_{\sg_1}^{\sg_2}(t)=\left\{
  \begin{array}{ll}
   1 & if\ t\geq \sg_2\\
  \frac{t-\sg_1}{\sg_2-\sg_1} & if\ \sg_1\leq t \leq \sg_2\\
  0 & if\ 0<t\leq \sg_1.
  \end{array} \right.
 $$






\begin{definition}{\rm(\cite{HLX}) }\label{def:2.3} Let  $\Omega$ be a class of unital
$\rm C^*$-algebras. A unital $\rm C^*$-algebra $A$ is said to have property ${\rm
(I)}$ if,  for any positive numbers $0<\sg_3<\sg_4<\sg_1<\sg_2<1$,
any $\varepsilon>0$, any finite subset ${F}\subseteq A$,  any finite nonzero positive element $G\subseteq A_+$,  and any integer $n>0$, there exist a
nonzero projection $p\in A$ and a $C^*$-subalgebra $B$ of $A$ with $B\in \Omega$
and $1_B=p$, such that

$(1)$ $\|xp-px\|<\varepsilon$, for all $x\in{F} $,

$(2)$ $pxp\in_\varepsilon B$, for all $x\in{F}$,

$(3)$
$n[f^{\sg_2}_{\sg_1}((1-p)a(1-p))]\leq[f^{\sg_4}_{\sg_3}(pap)]$, for any $a\in G$.
\end{definition}

\begin{theorem}{ \rm( \cite{HLX}) }\label{thm:2.4} Let $A$ be a unital simple $\rm {C^*}$-algebra.  Then $A$ belong to  ${\rm TA}\Omega$  if and only if $A$ has property ${\rm
(I)}$.
\end{theorem}

\begin{theorem}\label{thm:2.5}
		Let $\Omega$ be a class of simple unital $\rm C^{*}$-algebra. Let $A$ be a unital simple $\rm C^{*}$-algebra. Then  $A\in \rm {TA}\Omega$, if and only if,
		for any $\varepsilon>0$, any finite $F\subseteq A$, and any nonzero positive element $a\in A_{+}$, there exist a nonzero
		projection $p\in A$ and a $\rm C^{*}$-subalgebra $B$ of $A$ with $1_{B}=p$ and $B\in\Omega$, such that
		
		$(1)$ $\|px-xp\|<\varepsilon$, for all $x\in F$,
		
		$(2)$ $pxp\in_{\varepsilon}B$,  for all $x\in F$,
		
		$(3)$ $1-p\precsim a$, and
		
		$(4)$ $\|pxp\|\geq\|x\|-\varepsilon$, for all $x\in F$.
	\end{theorem}
    \begin{proof}
    	$\Longleftarrow$: This is obvious.
    	
    	$\Longrightarrow$: We must show that for any $\varepsilon>0$, any finite $F=\{x_{1},~\cdots,~ x_{n}\}\subseteq A$ (we may assume that $\|x_i\|=1$, for any $1\leq i\leq n$),	any nonzero positive element $a\in A_{+}$, there exist a nonzero
    	projection $p\in A$ and a $\rm C^{*}$-subalgebra $B$ of $A$ with $1_{B}=p$, and $B\in\Omega$, such that
    	
    	$(1)$ $\|px-xp\|<\varepsilon$, for all $x\in F$,
    	
    	$(2)$ $pxp\in_{\varepsilon}B$, for all $x\in F$,
    	
    	$(3)$ $1-p\precsim a$, and
    	
    	$(4)$ $\|pxp\|\geq\|x\|-\varepsilon$, for all $x\in F$.
    	
    	Since $A\in \rm {TA}\Omega$, with  $\varepsilon>0$,  and $G=\{x_{1},~\cdots,~x_{n},~x_{1}^{*}x_{1},~\cdots,$ $~x_{n}^{*}x_{n}\}\subseteq A$,
    	there exist a nonzero
    	projection $p\in A$, and a $\rm C^{*}$-subalgebra $B$ of $A$ with $1_{B}=p$, and $B\in \Omega$, such that
    	
    	$(1')$ $\|px-xp\|<\varepsilon/3$, for all $x\in G$,
    	
    	$(2')$ $pxp\in_{\varepsilon/3}B$, for all $x\in G$, and
    	
    	$(3')$ $1-p\precsim a$.
    	
    	By Theorem \ref{thm:2.4}, for  $x_{i}^{*}x_{i}\in A_{+}$,  with $\|x_{i}^{*}x_{i}\|=1$ ($1\leq i\leq n$), let
    	$0<\sigma_{3}=1-4\varepsilon<\sigma_{4}=1-3\varepsilon<\sigma_{1}=1-\varepsilon<\sigma_{2}<1$, we have
    	
    	$(4')$ $[f^{\sigma_{2}}_{\sigma_{1}}((1-p)x_{i}^{*}x_{i}(1-p))]\leq[f^{\sigma_{4}}_{\sigma_{3}}(px_{i}^{*}x_{i}p)]$, for $1\leq i\leq n$.
    	
    	If $\|px_{i}^{*}x_{i}p\|<1-\varepsilon$, then we have $f^{\sigma_{4}}_{\sigma_{3}}(px_{i}^{*}x_{i}p)=0$, by $(4')$, one has

    $$f^{\sigma_{2}}_{\sigma_{1}}((1-p)x_{i}^{*}x_{i}(1-p))=0.$$ So we has $$\|(1-p)x_{i}^{*}x_{i}(1-p)\|<1-4\varepsilon.$$

    	By $(1')$ and $(2')$, for $i=1,~\cdots,~n$, we have
    	
    	$\|x_{i}^{*}x_{i}-(1-p)x_{i}^{*}x_{i}(1-p)-px_{i}^{*}x_{i}p\|$
    	
    	$\leq\|x_{i}^{*}x_{i}-(1-p)x_{i}^{*}x_{i}-px_{1}^{*}x_{i}\|$

    $+\|(1-p)x_{i}^{*}x_{i}-(1-p)x_{i}^{*}x_{i}(1-p)\|+
    	\|px_{i}^{*}x_{i}-px_{i}^{*}x_{i}p\|$
    	
    	$<\varepsilon/3+\varepsilon/3+\varepsilon/3=\varepsilon$.
    	
    	Then, we have

     $1=\|x_{i}^{*}x_{i}\|\leq \|(1-p)x_{i}^{*}x_{i}(1-p)+px_{i}^{*}x_{i}p\|+\varepsilon$

     $= \max\{\|(1-p)x_{i}^{*}x_{i}(1-p),~\|px_{i}^{*}x_{i}p\|\}+\varepsilon <1$, this is  a contradictory.
    	
    	Therefore, for any $1\leq i\leq n$, we have  $$\|px_{i}^{*}x_{i}p\|\geq1-\varepsilon.$$
    	
    	Since $\|px_ip\|\leq\|x\|=1$, then, for any $1\leq i\leq n$, we have

     $$\|px_ip\|\geq\|px_{i}p\|^{2}\geq\|px_{i}^{*}x_{i}p\|-\varepsilon\geq1-2\varepsilon.$$
    	
    	With $2\varepsilon$ in place of $\varepsilon$.
    \end{proof}

Winter and Zacharias  introduce  nuclear dimension for  ${\rm C^*}$-algebras in \cite{WW3}.

\begin{definition}{\rm (\cite{WW3}.)}\label{def:2.6}
Let $A$ be a  ${\rm C^*}$-algebra, $m\in {\mathbb{N}}$.
 A complete positive compression   $\varphi:F\to A$ is $m$-decomposable (where $F$ is
finite dimensional  ${\rm C^*}$-algebra ), if there is  a decomposition $$F=F^{(0)}\oplus F^{(1)}
\oplus\cdots \oplus F^{(m)}$$ such that the restriction $\varphi^{(i)}$ of $\varphi$ to $F^{(i)}$
has order zero (A complete positive map $\varphi:A\to B$ is order zero if for any positive elements  $a,b$ with $ab=0$ we have $\varphi(a)\varphi(b)=0$)  for each $i\in \{0,\cdots, m\}$, we say $\varphi$ is
$m$-decomposable with respect to $F=F^{(0)}\oplus F^{(1)}
\oplus\cdots \oplus F^{(m)}$.

 A has nuclear dimension $m$,  write ${\rm {dim_{nuc}}}(A)=m$, if $m$ is the least integer such that
the following holds: For any finite subset $\mathcal{G}\subseteq A$ and $\varepsilon>0$, there is
a finite dimension complete positive compression  approximation $(F,\varphi, \psi)$
for $\mathcal{G}$ to within $\varepsilon$ (i.e., $F$ is finite dimensional $\psi: A\to F$  and
$\varphi:F\to A$ are complete positive  and $\|\varphi\psi(b)-b\|<\varepsilon$
for any $b\in \mathcal{G}$) such that $\psi$ is  complete positive compression, and $\varphi$ is $m$-decomposable with
 complete positive compression order zero components $\varphi^i$. If no such $m$ exists, we write
$dim_{nuc}(A)=\infty$.
\end{definition}

Inspired by  Hirshberg and  Orovitz's  tracial $\mathcal{Z}$-absorption in \cite{HO}, Fu introduced
 some notion of  tracial nuclear dimension in his doctoral dissertation  \cite{FU} (see also  \cite{FL}).

	\begin{definition}{\rm (\cite{FL})}\label{def:2.7}
		Let $A$ be a unital simple $\rm C^{*}$-algebra.  Let $n\in \mathbb{N}$ be an integer. $A$ is said to have tracial nuclear dimension at most $n$,
		and denoted by ${\rm Tdim_{nuc}} A\leq n$, if for any finite subset $\mathcal{F}\subseteq A$, for any $\varepsilon>0$ and for any nonzero positive element $a\in A_{+}$, there exist a finite dimensional $\rm C^{*}$-algebra $F$, a  completely positive contractive  map $\alpha:A\rightarrow F$,
		a nonzero piecewise contractive $n$-decomposable completely positive  map $\beta:F\rightarrow A$, and a completely positive contractive map
		$\gamma:A\rightarrow A\cap\beta^{\perp}(F)$ such that
		
		 $(1)$ $\|x-\gamma(x)-\beta\alpha(x)\|<\varepsilon$, for all $x\in\mathcal{F}$, and
		
		 $(2)$ $\gamma(1_{A})\precsim a$.
	\end{definition}
Note that: Tracial nuclear dimension at most $n$ is closed under tensoring with matrix algebras and under passing to unital hereditary $\rm C^{*}$-subalgebras.

\begin{theorem}\label{thm:2.9} \label{thm:2.8} Let $A$ be a unital
simple ${\rm C^*}$-algebra. Then   ${\rm Tdim_{nuc}} A\leq n$, if and only if, if for any finite subset $\mathcal{F}\subseteq A$, for any $\varepsilon>0$ and for any nonzero positive element $a\in A_{+}$, there exist a finite dimensional $\rm C^{*}$-algebra $F$, a  completely positive contractive  map $\alpha:A\rightarrow F$,
		a nonzero piecewise contractive $n$-decomposable completely positive  map $\beta:F\rightarrow A$, and a completely positive contractive map
		$\gamma:A\rightarrow A\cap\beta^{\perp}(F)$ such that
		
		 $(1)$ $\|x-\gamma(x)-\beta\alpha(x)\|<\varepsilon$, for all $x\in\mathcal{F}$, and
		
		 $(2)$ $(\gamma(1_{A})-\varepsilon)_+\precsim a$.

\end{theorem}
\begin{proof}
As in the proof of Theorem \ref{thm:2.12}, we can arrange that $\gamma_{n}(1_{A})$ is a projection. Then, we have  $(\gamma_{n}(1_{A})-\varepsilon)_+\thicksim \gamma_{n}(1_{A})$.
\end{proof}

\begin{definition}{\rm (\cite{FU})}\label{def:2.9}
	Let $A$ be a $\rm C^{*}$-algebra. Let $n\in \mathbb{N}$ be an integer. $A$ is said to have the second type tracial nuclear dimension
	at most $n$, and denoted by ${\rm T^{2}dim_{nuc}} A\leq n$, if for any finite positive subset $\mathcal{F}\subseteq A$, for any $\varepsilon>0$ and for any nonzero positive element  $a\in A_{+}$, there exist a finite dimensional $\rm C^{*}$-algebra $F=F_{0}\oplus\cdots\oplus F_{n}$ and completely positive  maps $\psi:A\rightarrow F$, $\varphi:F\rightarrow A$ such that
	
	$(1)$  for any $x\in F$, there exists $x'\in A_{+}$, such that $x'\precsim a$, and $\|x-x'-\varphi\psi(x)\|<\varepsilon$,
	
	$(2)$ $\|\psi\|\leq1$, and
	
	$(3)$ $\varphi|_{F_{i}}$ is a completely positive contractive  order zero map, for any $0\leq i\leq n$.
    \end{definition}

Let $A$ be a unital $\rm C^{*}$-algebra. It is easy to know that
 ${\rm {Tdim_{nuc}}}(A)\leq n$ implies  that ${\rm {T^2dim_{nuc}}}(A)\leq n$.

Note that: Second type tracial nuclear dimension  at most $n$ is closed under tensoring with matrix algebras and under passing to unital hereditary $\rm C^{*}$-subalgebras.

 The tracial Rokhlin property was  introduced by Phillips  in \cite{P2} for finite group actions.

 \begin{definition}{\rm (\cite{P2}.)} \label{def:2.10} Let $A$ be a separable simple  unital
 ${\rm C^*}$-algebra,
  and let $\alpha: G\to {\rm Aut}(A)$  be an action of  a finite group $G$ on $A$. We say that $\alpha$ has the
tracial Rokhlin property if, for any  finite set $\mathcal{G}\subseteq A,$
any  $\varepsilon>0,$ and  any  non-zero positive element $b\in A,$   there are  mutually orthogonal projections $e_g\in A$ for
$g\in G$ such that

$(1)$ $\|{\alpha_g}(e_h)-e_{gh}\|<\varepsilon$, for all  $g,~ h\in G,$

$(2)$ $\|e_gd-de_g\|<\varepsilon $, for all $g\in G$, and all $d\in { \mathcal{G}},$

$(3)$ with $e=\Sigma _{g\in G}e_g,$ the projection $1-e$ is Murray-von Neumann equivalent to a projection in the hereditary ${\rm C^*}$-subalgebra of $A$ generated by $b$, and

$(4)$ $\|ebe\|\geq \|b\|-\varepsilon$.
\end{definition}

	Let  $\Omega$ be a class of unital ${\rm C^*}$-algebras.  Then  the class
 of ${\rm C^*}$-algebras which are asymptotical   tracially   in $\Omega$, denoted  by ${\rm AT}\Omega$ was introduced by Fu and Lin in \cite{FL}.

    \begin{definition}\label{def:2.11}{\rm (\cite{FL})}
    	Let $A$ be a unital simple $\rm C^{*}$-algebra and  let $\Omega$ be a class of $\rm C^{*}$-algebras. We say $A$ is asymptotical tracially in $\Omega$, if for any finite subset $\mathcal{F}\subseteq A$, for any $\varepsilon>0$ and for any nonzero positive element $a\in A_{+}$, there exist a $\rm C^{*}$-algebra $B$ in $\Omega$, completely positive contractive  maps $\alpha:A\rightarrow B$,
        $\beta_{n}:B\rightarrow A$, and $\gamma_{n}:A\rightarrow A\cap\beta_{n}^{\perp}(B)(n\in \mathbb{N})$ such that

        $(1)$ $\|x-\gamma_{n}(x)-\beta_{n}\alpha(x)\|<\varepsilon$, for all $x\in\mathcal{F}$, and for all $n\in{\mathbb{N}}$,

        $(2)$ $\alpha$ is a $(\mathcal{F},\varepsilon)$-approximate embedding,

        $(3)$ $\mathop{\lim}\limits_{n\rightarrow\infty}\|\beta_{n}(xy)-\beta_{n}(x)\beta_{n}(y)\|=0$ and
        $\mathop{\lim}\limits_{n\rightarrow\infty}\|\beta_{n}(x)\|=\|x\|$, for all $x,~y\in B$, and

        $(4)$ $\gamma_{n}(1_{A})\precsim a$, for all $n\in \mathbb{N}$.
    \end{definition}

\begin{theorem}\label{thm:2.12}{\rm (\cite{FL})} Let $\Omega$ be a class of unital ${\rm C^*}$-algebras  such that $\Omega$ is closed under tensoring with matrix algebras and
 under passing to unital
hereditary ${\rm C^*}$-subalgebras. Let $A$ be a unital
simple ${\rm C^*}$-algebra. Then $A$  is asymptotical tracially in $\Omega$ (i.e. $A\in{\rm AT}\Omega$), if and only if, for any
 $\varepsilon>0,$ any finite
subset $\mathcal{G}\subseteq A,$ and any  non-zero element $a\geq 0,$ there exist
a ${\rm C^*}$-algebra $D$  in $\Omega$ and a nonzero  completely positive  contractive linear map  $\alpha:A\to D$ and  completely positive  contractive linear maps $\beta_n: D\to A$, $\gamma_n:A\to A\cap\beta_n(D)^{\perp}$ such that

$(1)$ the map $\alpha$ is unital  completely positive  linear map, $\beta_n(1_D)$ and $\gamma_n(1_A)$ are projections and  $\beta_n(1_D)+\gamma_n(1_A)=1_A$, for all $n\in \mathbb{N}$,

$(2)$ $\|x-\gamma_n(x)-\beta_n\alpha(x)\|<\varepsilon$, for all $x\in \mathcal{G}$, and for all $n\in {\mathbb{N}}$,

$(3)$ $\alpha$ is a $\mathcal{G}$-$2\varepsilon$ approximate embedding,

$(4)$ $\lim_{n\to \infty}\|\beta_n(xy)-\beta_n(x)\beta_n(y)\|=0$ and $\lim_{n\to \infty}\|\beta_n(x)\|=\|x\|$, for all $x,~y\in D$, and

$(5)$ $\gamma_n(1)\precsim a$, for all $n\in \mathbb{N}$.
\end{theorem}

\begin{theorem}\label{thm:2.13} Let $\Omega$ be a class of unital ${\rm C^*}$-algebras  such that $\Omega$ is closed under tensoring with matrix algebras and
 under passing to unital
hereditary ${\rm C^*}$-subalgebras. Let $A$ be a unital
simple ${\rm C^*}$-algebra. Then $A$  is asymptotical tracially in $\Omega$ (i.e. $A\in{\rm AT}\Omega$), if and only if,  for any finite subset $\mathcal{F}\subset A$, for any $\varepsilon>0$ and for any nonzero positive element $a\in A_{+}$, there exist a $\rm C^{*}$-algebra $B$ in $\Omega$, completely positive contractive  maps $\alpha:A\rightarrow B$,
        $\beta_{n}:B\rightarrow A$, and $\gamma_{n}:A\rightarrow A\cap\beta_{n}^{\perp}(B)(n\in \mathbb{N})$ such that

        $(1)$ $\|x-{\varepsilon}\gamma_{n}(x)-\beta_{n}\alpha(x)\|<\varepsilon$, for all $x\in\mathcal{F}$, and for all $n\in{\mathbb{N}}$,

        $(2)$ $\alpha$ is a $(\mathcal{F},2\varepsilon)$-approximate embedding,

        $(3)$ $\mathop{\lim}\limits_{n\rightarrow\infty}\|\beta_{n}(xy)-\beta_{n}(x)\beta_{n}(y)\|=0$ and
        $\mathop{\lim}\limits_{n\rightarrow\infty}\|\beta_{n}(x)\|=\|x\|$, for all $x,~y\in B$,and

        $(4)$ $(\gamma_{n}(1_{A})-\varepsilon)_+\precsim a$, for all $n\in \mathbb{N}$.
\end{theorem}
\begin{proof}
By Theorem \ref{thm:2.12}, we can arrange that $\gamma_{n}(1_{A})$ is a projection. Then, we have  $(\gamma_{n}(1_{A})-\varepsilon)_+\thicksim \gamma_{n}(1_{A})$.
\end{proof}

    \begin{lemma}{\rm (\cite{FL}.)}\label{lem:2.14} If the class $\Omega$ is closed under tensoring with matrix algebras and
 under passing to  unital
hereditary ${\rm C^*}$-subalgebras, then  the class of  asymptotically tracially in $\Omega$  is closed under
tensoring with matrix algebras  and under passing to unital hereditary ${\rm C^*}$-subalgebras.
\end{lemma}

The following theorem is Theorem 2.3.13 in \cite{L2}.

\begin{theorem}\label{thm:2.15}{\rm (\cite{L2}.)}
Let $A$ be  ${\rm C^*}$-algebra, $B$ be a ${\rm C^*}$-subalgebra of $A$ and $C$
be another ${\rm C^*}$-algebra. Suppose that $\phi:B\to C$ is a contractive completely positive linear map. If either $B$ or $C$ is amenable, the for any finite subset $\mathcal{G}\subset B$ and $\varepsilon>0$, there is a contractive completely positive linear map $\widetilde{\phi}:A\to C$  such that $$\|\widetilde{\phi}-\phi\|<\varepsilon~~ \rm {on}~ \mathcal{G}.$$
\end{theorem}

\begin{definition} \label{def:2.16}Let $A$ and $B$ be  ${\rm C^*}$-algebra,  let $\varphi:A\to B$
be a map, let $\mathcal{G}\subset A$, and let $\varepsilon>0$. The map $\varphi$ is called $\mathcal{G}$-$\varepsilon$-multiplicative, or called $\varepsilon$-multiplicative on $\mathcal{G}$, if
for any $x,y\in F$, $\|\varphi(xy)-\varphi(x)\varphi(y)\|<\varepsilon$. If, in addition, for any
$x\in \mathcal{G}$, $|\|\varphi(x)\|-\|x\||<\varepsilon$, then we say $\varphi$ is an $\mathcal{G}$-$\varepsilon$-approximate embedding.
\end{definition}

\begin{definition} \label{def:2.17} {\rm (\cite{FL}.)} Let $A$ and $B$ be  ${\rm C^*}$-algebra and let
 $\varphi:A\to B$ be a map. let $\varepsilon>0$. If, for any $a_1,a_2\in A_+^1$ with $a_1a_2=0$, we have $\|\varphi(a_1)\varphi(a_2)\|\leq \varepsilon$, then we say $\varphi$ is an $\varepsilon$-almost order zero map.
 \end{definition}

\begin{definition}\label{def:2.18} {\rm (\cite{FL}.)} Let $A$  be a  ${\rm C^*}$-algebra and let $F$ be a finite dimensional  ${\rm C^*}$-algebra. Let
$\varphi:F\to A$ be an integer. The map $\varphi$ is called $(n-\varepsilon)$-dividable if $F$ can be written as $F=F_0\oplus F_1\oplus\cdots \oplus F_n$ (where $F_i$ are ideals of $F$) such that $\varphi|F_i$ is a completely
 \end{definition}

    \begin{theorem}{\rm (\cite{FL}.)}\label{thm:2.19}
    	For any finite dimension $\rm C^{*}$-algebra $F$ and any $\varepsilon>0$,there exists $\delta>0$ such that,
    	for any $\rm C^{*}$-algebra $A$ and any completely positive  contractive map $\varphi:F\rightarrow A$ which is $\delta$-almost order zero,
    	there exist a completely positive  contractive order zero map $\psi:F\rightarrow A$ satisfying $\|\varphi-\psi\|\leq\varepsilon$.
    \end{theorem}
	
	\section{the main result}
	
	\begin{theorem} \label{thm:3.1}Let $\Omega$ be a class of  unital
$\rm {C^*}$-algebras.  Then  $A$ is asymptotically tracially in $\Omega$ for  any   unital  simple separable infinite dimensional amenable $\rm {C^*}$-algebra $A\in {{\rm TA}}({\rm AT}\Omega)$.
\end{theorem}

\begin{proof} We need to show that for any
$\varepsilon>0$, any finite
subset  $\mathcal{G}=\{a_1,a_2, $ $\cdots, a_k\}\subseteq A$ (we may assume that $\|a_i\|\leq 1$ for all $1\leq i\leq k$), and any  nonzero positive  element $b$, there exist
a ${\rm C^*}$-algebra $B$  in $\Omega$ and completely positive  contractive linear maps  $\alpha:A\to B$ and  $\beta_n: B\to A$, and $\gamma_n:A\to A$ such that

$(1)$  $\|x-\gamma_n(x)-\beta_n\alpha(x)\|<\varepsilon$, for all $x\in \mathcal{G}$, and for all $n\in {\mathbb{N}}$,

$(2)$ $\alpha$ is a $\mathcal{G}$-$\varepsilon$ approximate embedding,

$(3)$ $\lim_{n\to \infty}\|\beta_n(xy)-\beta_n(x)\beta_n(y)\|=0$ and $\lim_{n\to \infty}\|\beta_n(x)\|=\|x\|$, for all $x,~y\in B$, and

$(4)$ $\gamma_n(1)\precsim a$, for all $n\in \mathbb{N}$.

Since $A$ is an infinite-dimensional simple unital $\rm {C^*}$-algebra there exist nonzero positive elements  $b_1, ~b_2\in A_+$, such that
$b_1b_2=0$ and $b_1+b_2\precsim b$.

Since   $A\in {{\rm TA}}({\rm AT}\Omega)$, by Theorem \ref{thm:2.5}, for  any $\delta>0$ (with $\delta<\varepsilon$),
any finite subset $\mathcal{H}=\{a_1,~a_2, $ $\cdots, a_k,~ b_2, ~a_ia_j,~1\leq i,~j\leq k\}\subseteq A$ , non-zero positive element $b_1$ of $A$,   there exist  a  projection $p\in A$ and a ${\rm C^*}$-subalgebra $B$ of $A$ with
$1_B=p$ and $B\in{\rm AT}\Omega$, such that

$(1)'$ $\|b_2p-pb_2\|<\delta$, $\|a_ip-pa_i\|<\delta$, $\|a_ia_jp-pa_ia_j\|<\delta$, for all $1\leq i,~j\leq k$,

$(2)'$ $pb_2p\in_\delta B$, $pa_ip\in_\delta B$, $pa_ia_jp\in_\delta B$, for
all  $1\leq i,~j\leq k$,

 $(3)'$ $1-p\precsim b_1$, and

 $(4)'$ $\|pa_ip\|\geq \|a_i\|-\varepsilon$, $\|pa_ia_jp\|\geq \|a_ia_j\|-\varepsilon$,  for all  $1\leq i,~j\leq k$.

By $(2)'$  there exist $a_i',~ b_2'\in B$ such that
 $$\|pa_ip-a_i'\|<\delta, \|b_2-b_2'\|<\delta$$
for all $1\leq i\leq k$.
Since $A$ is a amenable $\rm {C^*}$-algebra and $B\subseteq A$, by Theorem  \ref{thm:2.15}, there exists
a   completely positive  contractive linear map $\alpha':A\to B$ such that
$$\|\alpha'(a_i')-a_i'\|<\delta,~\|\alpha'(a_i'a_j')-a_i'a_j\|<\delta $$ for all $1\leq i, ~j\leq k$.

Since $B\in \rm {AT}\Omega$, by Theorem \ref{thm:2.12},  for $\mathcal{H'}=\{a_1', a_2',\cdots, a_k'\}$, any $\delta''>0$ (with $\delta''<\delta$), positive element $(b_2'-\delta)_+$ (we can assume that $(b_2'-\delta)_+\neq 0$),
 there exist
a ${\rm C^*}$-algebra $D$  in $\Omega$ and  completely positive  contractive linear maps  $\alpha'':B\to D$ and  completely positive  contractive linear maps $\beta_n'': D\to B$, and $\gamma_n'':B\to B\cap\beta_n''(D)^{\perp}$ such that

$(1)''$ the map $\alpha''$ is unital  completely positive   linear map,  $\beta_n''(1_D)$ and $\gamma_n''(1_B)$ are projections and  $\beta_n''(1_D)+\gamma_n''(1_B)=1_B$, for all $n\in \mathbb{N}$,

$(2)''$ $\|a_i'-\gamma_n''(a_i')-\beta_n''(\alpha''(a_i'))\|<\delta''$, for all $1\leq i\leq k$, and for all $n\in {\mathbb{N}}$,

$(3)''$ $\alpha''$ is an $\mathcal{H'}$-$\delta''$ approximate embedding,

$(4)''$ $\lim_{n\to \infty}\|\beta_n''(xy)-\beta_n''(x)\beta_n''(y)\|=0$, and $\lim_{n\to \infty}\|\beta_n''(x)\|=\|x\|$, for all $x,~y\in D$, and

$(5)''$ $\gamma_n''(1_B)\precsim (b_2'-\delta)_+$, for all $n\in \mathbb{N}$.

 Define  $\gamma_n:A\to A$, by $\gamma_n(a)=(1-p)a(1-p)+\gamma_n''\alpha'(pap)$,    $\alpha: A\to D$, by $\alpha(a)=\alpha''\alpha'(pap)$ and  $\beta_n:D\to A$ by $\beta_n(a)=\beta_n''(a)$. Then we have   $\gamma_n$,  $\beta_n$ and $\alpha$ are  completely positive  contractive linear maps (since for any $a\in A$, $(1-p)a(1-p)\gamma_n''\alpha'(pap)=0$).

 We also have
$$\begin{array}{ll}
&$(1)$~~~~~\|a_i-\gamma_n(a_i)-\beta_n\alpha(a_i)\|\\
&= \|a_i-(1-p)a_i(1-p)-\gamma''\alpha'(pa_ip)-\beta_n''\alpha'' \alpha'(pa_ip)\|\\
&\leq \|a_i-(1-p)a_i(1-p)-pa_ip\|+\|pa_ip-\gamma''\alpha'(pa_ip)-\beta_n''\alpha'' \alpha'(pa_ip)\|\\
&\leq 2\delta+\|pa_ip-a_i'\|+\|a_i'-\gamma''\alpha'(pa_ip)-\beta_n''\alpha'' \alpha'(pa_ip)\|\\
&=2\delta+\delta+\|a_i'-\gamma''(a_i')-\beta_n''\alpha'' (a_i')\|+\|\gamma''\alpha'(pa_ip)- \gamma''\alpha'(a_i')\|\\
&+\|\gamma''\alpha'(a_i')- \gamma''(a_i')\|
+\|\beta_n''\alpha'' \alpha'(pa_ip)-\beta_n''\alpha'' \alpha'(a_i')\|\\
 &+\|\beta_n''\alpha'' \alpha'(a_i')-\beta_n''\alpha'' (a_i')\|<3\delta+\delta''+\delta+\delta+\delta+\delta<\varepsilon,
\end{array}$$
for all $1\leq i,j\leq k$.

 $$\begin{array}{ll}
&$(2)$~~~~~\|\alpha(a_ia_j)-\alpha(a_i)\alpha(a_j)\|\\
&= \|\alpha''\alpha'(pa_ia_jp)-\alpha''\alpha'(pa_ip)\alpha''\alpha'(pa_jp)\|\\
&\leq \|\alpha''\alpha'(pa_ia_jp)-\alpha''\alpha'(a_i'a_j')\|
+\|\alpha''\alpha'(a_i'a_j')-\alpha''(\alpha'(a_i')\alpha'(a_j'))\|\\
&+\|\alpha''(\alpha'(a_i')\alpha'(a_j'))-\alpha''\alpha'(a_i')\alpha''\alpha'(a_j')\|\\
&+\|\alpha''\alpha'(a_i')\alpha''\alpha'(a_j')-\alpha''\alpha'(pa_ip)
\gamma''\alpha'(pa_jp)\|\\
&<\delta+\delta+2\delta+\delta<\varepsilon,
\end{array}$$
 for all $1\leq i,j\leq k$.

 We also  have $\|\alpha(a_i)\|=\|\alpha''\alpha'(pa_ip)\|\geq \|\alpha'(pa_ip)\|-\delta''\geq \|\alpha'(a_i)\|-\delta''-\delta\geq\|a_i'\|-\delta''-2\delta\geq\|pa_ip\|-\delta''
 -3\delta\geq\|a_i\|-\delta''-4\delta>\|a_i\|-\varepsilon.$

 Therefore, $\alpha$ is a  $\mathcal{G}$-$\varepsilon$ approximate embedding.

 $(3)$ $\lim_{n\to \infty}\|\beta_n(xy)-\beta_n(x)\beta_n(y)\|=0$ and $\lim_{n\to \infty}\|\beta_n(x)\|=\|x\|$, for all $x,~y\in D$, and

 $(4)$ $\gamma_n(1_A)=1-p+\gamma_n''\alpha'(p1_Ap)\precsim b_1\oplus\gamma_n''(1_B)\precsim b_1\oplus(b_2'-\delta)_+\precsim b_1+b_2\precsim b$, for all $n\in \mathbb{N}$.
\end{proof}

	\begin{theorem} \label{thm:3.2}
		Let $\Omega$ be a class of unital $\rm C^{*}$-algebras which have tracial nuclear dimension at most $n$. Let  $A$ be an infinite dimensional  simple unital $\rm C^{*}$-algebra and $A$   is asymptotical tracially in $\Omega$ (i.e. $A\in \rm AT\Omega$). Then ${\rm Tdim_{nuc}}(A)\leq n$.	
	\end{theorem}
	\begin{proof} By Theorem \ref{thm:2.8}, we must show that for any finite subset $G\subseteq A$, for any $\varepsilon>0$, and for any nonzero positive element $a\in A_{+}$, there exist a finite dimensional $\rm C^{*}$-algebra $F$,  a completely positive contractive  map $\alpha: A\rightarrow F$,
	    a nonzero piecewise contractive $n$-decomposable completely positive  map $\beta:F\rightarrow A$, and a completely positive contractive map
	    $\gamma:A\rightarrow A\cap\beta^{\perp}(F)$ such that
	
	    $(1)$ $\|x-\gamma(x)-\beta\alpha(x)\|<\varepsilon$, for all $x\in G$, and
	
	    $(2)$ $(\gamma(1_{A})-\varepsilon)_+\precsim a$.
	
	    Since $A$ is an infinite dimensional simple unital  $\rm C^{*}$-algebra, there exist positive elements $a_{1},~a_{2}\in A_{+}$ of norm one such that $a_{1}a_{2}=0$,
	    $a_{1}\sim a_{2}$ and $a_{1}+a_{2}\precsim a$.
	
	    Give $\varepsilon'>0$,  with $G'=G\cup\{a_{1},~a_{2},~1_{A}\}\subseteq A$, since $A$   is asymptotical tracially in $\Omega$ (i.e., $A\in \rm AT\Omega$),
	    there exist a $\rm C^{*}$-algebra $B$ in ${\Omega}$, completely positive contractive maps $\overline{\alpha}:A\rightarrow B$,
	    $\overline{\beta_{n}}:B\rightarrow A$, and $\overline{\gamma_{n}}:A\rightarrow A\cap\beta_{n}^{\perp}(B)(n\in \mathbb{N})$ such that
	
	    $(1')$ $\|x-\overline{\gamma_{n}}(x)-\overline{\beta_{n}}\overline{\alpha}(x)\|<\varepsilon'$, for all $x\in G'$, and for all $n\in \mathbb{N}$,
	
	    $(2')$ $\overline{\alpha}$ is a $(G',~\varepsilon')$-approximate embedding,
	
	    $(3')$ $\mathop{\lim}\limits_{n\rightarrow\infty}\|\overline{\beta_{n}}(xy)-\overline{\beta_{n}}(x)\overline{\beta_{n}}(y)\|=0$ and
	    $\mathop{\lim}\limits_{n\rightarrow\infty}\|\overline{\beta_{n}}(x)\|=\|x\|$ for all $x,~y\in B$, and
	
	    $(4')$ $\overline{\gamma_{n}}(1_{A})\precsim a_{1}\sim a_{2}$, for all $n\in \mathbb{N}$.
	
	    Since $a_{2}\in A_{+}$ and $\overline{\alpha}$ is a $(G',\varepsilon')$-approximate embedding, then
	    $\overline{\alpha}(a_{2})\in B_{+}$ and $\overline{\alpha}(a_{2})\neq 0$. We may assume that $(\overline{\alpha}(a_{2})-\varepsilon')_+\neq 0$.
	
	    Since $B\in{\Omega}$,  ${\rm Tdim_{nuc}}(B)\leq n$, with  $\{\overline{\alpha}(x), x\in G'\}\subseteq B$, with $\varepsilon'$, and nonzero positive element  $(\overline{\alpha}(a_{2})-\varepsilon')_+\in B_{+}$, there exist a
	    finite dimensional $\rm C^{*}$-algebra $F$,  a completely positive contractive map $\alpha':B\rightarrow F$, a nonzero piecewise contractive $n$-decomposable completely positive  map $\beta':F\rightarrow B$,  and a completely positive contractive map
	    $\gamma':B\rightarrow B\cap\beta^{\perp}(F)$ such that
	
	    $(1'')$ $\|\overline{\alpha}(x)-\gamma'\overline{\alpha}(x)-\beta'\alpha'\overline{\alpha}(x)\|<\varepsilon'$, for all $\{\overline{\alpha}(x), x\in G'\}$, and
	
	    $(2'')$ $(\gamma'(1_{B})-\varepsilon)_+\precsim (\overline{\alpha}(a_{2})-\varepsilon')_+$.
	
	    Define $\alpha:A\rightarrow F$, by $\alpha(x)=\alpha'\overline{\alpha}(x)$, $\varphi:F\rightarrow A$, by $\varphi(x)=\overline{\beta_{n}}\beta'(x)$.
	
	    For $x,~y\in F_{i}$ with $xy=0$, then $\varphi(x)\varphi(y)=\overline{\beta_{n}}\beta'(x)\overline{\beta_{n}}\beta'(y)$.
	    By $(3')$,  for $\delta>0$, there  exists $N>0$, such that  $n>N$, we have $\|\overline{\beta_{n}}\beta'(x)\overline{\beta_{n}}\beta'(y)-\overline{\beta_{n}}(\beta'(x)\beta'(y))\|<\delta$. Since $\beta'|_{F_{i}}$ is completely positive contractive order zero map, then $\varphi|_{F_{i}}$ is $\delta$-almost order zero map. By Theorem \ref{thm:2.19},
	    there exists $\beta:F\rightarrow A$ satisfying $\beta|_{F_{i}}$ is a completely positive contractive order zero map such that
	    $\|\beta-\varphi\|\leq\varepsilon$, then $\beta$ is a nonzero piecewise contractive $n$-decomposable completely positive  map.
	
	    Define $\gamma:A\rightarrow A\cap\beta^{\perp}(F)$, by
	    $\gamma(x)=\overline{\gamma_{n}}(x)+\overline{\beta_{n}}\gamma'\overline{\alpha}(x)$.
	Since $\overline{\gamma_{n}}(A)\perp\overline{\beta_{n}}(B)$, we have $\gamma$ is a completely positive contractive.

	    Then for all $x\in G$, we have
	
	    $\|x-\gamma(x)-\beta\alpha(x)\|=\|x-\overline{\gamma_{n}}(x)-\overline{\beta_{n}}\gamma'\overline{\alpha}(x)-\beta\alpha(x)\|$
	
	    $\leq\|x-\overline{\gamma_{n}}(x)-\overline{\beta_{n}}\gamma'\overline{\alpha}(x)-\varphi\alpha(x)\|+
	    \|\varphi\alpha(x)-\beta\alpha(x)\|$
	
	    $=\|x-\overline{\gamma_{n}}(x)-\overline{\beta_{n}}\gamma'\overline{\alpha}(x)-
	    \overline{\beta_{n}}\beta'\alpha'\overline{\alpha}(x)\|+\|\varphi\alpha(x)-\beta\alpha(x)\|$
	
	    $\leq\|x-\overline{\gamma_{n}}(x)-\overline{\beta_{n}}\overline{\alpha}(x)\|+
	    \|\overline{\beta_{n}}\overline{\alpha}(x)-\overline{\beta_{n}}\gamma'\overline{\alpha}(x)-
	    \overline{\beta_{n}}\beta'\alpha'\overline{\alpha}(x)\|$

 $+ \|\varphi\alpha(x)-\beta\alpha(x)\|$
	
	    $\leq\varepsilon'+\varepsilon'+\varepsilon'$
	
	    $\leq\varepsilon$.

Since $(\gamma'(1_{B})-\varepsilon)_+\precsim (\overline{\alpha}(a_{2})-\varepsilon')_+$,  for sufficiently small $\delta'>0$, there exist $v\in B$ such that

$$\|v(\overline{\alpha}(a_{2})-\varepsilon')_+v^*-(\gamma'(1_{B})-\varepsilon)_+\|<\delta',$$
 then, we have
 $$\|\overline{\beta_n}(v(\overline{\alpha}(a_{2})-\varepsilon')_+v^*)-\overline{\beta_n}((\gamma'(1_{B})-\varepsilon)_+)\|<\delta'.$$
By $(3')$, with sufficiently larger $n$, we have
 $$\|\overline{\beta_n}(v)\overline{\beta_n}((\overline{\alpha}(a_{2})-\varepsilon')_+)\overline{\beta_n}
(v^*)-\beta_n((\gamma'(1_{B})-\varepsilon)_+)\|<\delta'.$$

Therefore, we have
$$\overline{\beta_n}((\gamma'(1_{B})-2\varepsilon)_+)\precsim \overline{\beta_n}((\gamma'(1_{B})-3\delta'-\varepsilon)_+)\precsim  \overline{\beta_n}((\overline{\alpha}(a_{2})-\varepsilon')_+.$$

By $(1')$, $\|a_2-\overline{\gamma_n}(a_2)-\overline{\beta_n}\overline{\alpha}(a_2)\|<\varepsilon'$, then
we have $$(\overline{\beta_n}\overline{\alpha}(a_2)-\varepsilon')+\overline{\gamma_n}(a_2)\precsim a_2,$$ so we have $$(\overline{\beta_n}\overline{\alpha}(a_2)-\varepsilon')+\precsim a_2.$$

   We also have

$(\gamma(1_{A})-3\varepsilon)_+\precsim \overline{\gamma_{n}}(1_{A})+(\overline{\beta_{n}}\gamma'\overline{\alpha}(1_{A})-2\varepsilon)_+$

$\precsim
	    \overline{\gamma_{n}}(1_{A})+(\overline{\beta_{n}}\gamma'(1_{B})-2\varepsilon)_+$

$\precsim a_{1}\oplus\overline{\beta_n}(\overline{\alpha}(a_{2})-\varepsilon)_+$

$\precsim a_{1}+a_{2}\precsim a$.
	
	    Since $\alpha=\alpha'\overline{\alpha}$, $\alpha'$  and $\overline{\alpha}$ are completely positive contractive maps, then $\alpha$ is
	    a completely positive contractive map.
	
	    Therefore,  ${\rm Tdim_{nuc}}(A)\leq n$.	

\end{proof}

 \begin{theorem}\label{thm:3.3}{\rm (\cite{Q9}.)} Let $\Omega$ be a class of  unital
${\rm C^*}$-algebras such that $\Omega$ is closed under passing to  unital hereditary ${\rm C^*}$-subalgebra and tensoring matrix algebras. Let $A\in$ $\rm {TA}\Omega$  be  an infinite-dimensional  simple  unital ${\rm C^*}$-algebra.
 Suppose that $\alpha:G\to {\rm{Aut}}(A)$ is an action of a finite group $G$ on $A$
 which has  tracial Rokhlin property. Then the crossed product ${\rm C^*}$-algebra
 ${\rm C^*}( G, A,\alpha)$ belongs to $\rm{TA}\Omega$.
 \end{theorem}

The following Corollary also  can be obtained by a result of Fu and Lin of \cite{FL} and a result of Elliott, Fan and Fang of \cite{EFF} and Theorem \ref{thm:3.3}).

\begin{corollary}  \label{cor:3.4}Let $A$ be an infinite dimensional  simple separable amenable unital
 ${\rm C^*}$-algebra with  ${\rm Tdim_{nuc}}(A)\leq n$.
Suppose that  $\alpha:G\to {\rm Aut}(A)$ is  an action of a finite group $G$ on $A$
 which has the  tracial Rokhlin property.  Then ${\rm Tdim_{nuc}}({{\rm C^*}(G, A,\alpha)})\leq n$.
\end{corollary}

\begin{proof}
By Theorem \ref{thm:3.1}, Theorem \ref{thm:3.2} and Theorem \ref{thm:3.3}.

\end{proof}
	
    \begin{theorem} \label{thm:3.5}
    	Let $\Omega$ be a class of unital $\rm C^{*}$-algebras which have  the second type  tracial nuclear dimension at most $n$. Let $A$ be an infinite
    dimensional  simple unital   $\rm C^{*}$-algebra  and $A$   is asymptotical tracially in $\Omega$ (i.e. $A\in \rm AT\Omega$). Then  ${\rm T^2dim_{nuc}}(A)\leq n$.
    \end{theorem}
    \begin{proof}
    	We must show that for any positive subset $G\subseteq A_+$ (we may assume that $G\subseteq A_+^1$),  for any $\varepsilon>0$, and for any nonzero  positive element $a\in  A_{+}$, there exist a finite dimension $\rm C^{*}$-algebra $F=F_{0}\oplus\cdots\oplus F_{n}$ and completely positive  maps $\psi:A\rightarrow F$ and $\varphi:F\rightarrow A$ such that
    	
    	$(1)$ for all $x\in G$, exists $x'\in A_{+}$, such that $x'\precsim a$, and $\|x-x'-\varphi\psi(x)\|<\varepsilon$,
    	
    	$(2)$ $\|\psi\|\leq1$, and
    	
    	$(3)$ $\varphi|_{F_{i}}$ is  a completely positive contractive  order zero map, for $i=1,~\cdots,~n$.
    	
     Since $A$ is an infinite dimensional simple unital  $\rm C^{*}$-algebra, there exist positive elements $a_{1},~a_{2}\in A_{+}$ of norm one such that $a_{1}a_{2}=0$,
    	$a_{1}\sim a_{2}$ and $a_{1}+a_{2}\precsim a$.
    	
    	Give $\varepsilon'>0$, with $G'=G\cup\{a_{1},~a_{2}\}\subseteq A$, since $A$   is asymptotical tracially in $\Omega$ (i.e. $A\in \rm AT\Omega$),
    	there exist a $\rm C^{*}$-algebra $B$ in $\Omega$, and  completely positive contractive maps $\alpha:A\rightarrow B$,
    	$\beta_{n}:B\rightarrow A$, and $\gamma_{n}:A\rightarrow A\cap\beta_{n}^{\perp}(B)(n\in \mathbb{N})$ such that
    	
    	$(1')$ $\|x-\gamma_{n}(x)-\beta_{n}\alpha(x)\|<\varepsilon'$, for all $x\in G'$, and for all $n\in{\mathbb{N}}$,
    	
    	$(2')$ $\alpha$ is a $(G',\varepsilon')$-approximate embedding,
    	
    	$(3')$ $\mathop{\lim}\limits_{n\rightarrow\infty}\|\beta_{n}(xy)-\beta_{n}(x)\beta_{n}(y)\|=0$ and
    	$\mathop{\lim}\limits_{n\rightarrow\infty}\|\beta_{n}(x)\|=\|x\|$ for all $x,~y\in B$, and
    	
    	$(4')$ $\gamma_{n}(1_{A})\precsim a_{1}\sim a_{2}$, for all $n\in {\mathbb{N}}$.
    	
    	Since $\alpha$ is a $(G',\varepsilon')$-approximate embedding,  then
    	 $\alpha(a_{2})$ is nonzero. We may assume that $(\alpha(a_{2})-\varepsilon')_+\neq 0$.
    	
    	Since $B\in \Omega$, then ${\rm T^2dim_{nuc}}(B)\leq n$, with  $\{\alpha(x), x\in G'\}\subseteq B$, with $\varepsilon'$,
    	and nonzero positive element  $\alpha(a_{2})\in B_{+}$, there exist a finite dimension $\rm C^{*}$-algebra $F=F_{0}\oplus\cdots\oplus F_{n}$ and completely positive  maps $\psi':B\rightarrow F$, and $\varphi':F\rightarrow B$ such that
    	
    	$(1'')$ for all $x\in G'$, there exists $\alpha(x)'\in B_{+}$, such that $\alpha(x)'\precsim\alpha(a_{2})$ and
        and $\|\alpha(x)-\alpha(x)'-\varphi'\psi'\alpha(x)\|<\varepsilon'$,
    	
    	$(2'')$ $\|\psi'\|\leq 1$, and
    	
    	$(3'')$ $\varphi'|_{F_{i}}$ is a completely positive contractive  order zero map,  for $i=1,~\cdots,~n$.
    	
    	Define $\psi:A\rightarrow F$, by $\psi(a)=\psi'\alpha(a)$,  $\overline{\varphi}:F\rightarrow A$, by $\overline{\varphi}(a)=\beta_{n}\varphi'(a)$ and
    	$x'=\gamma_{n}(x)+(\beta_{n}\alpha(x)'-2\varepsilon)_{+}$, since $\gamma_{n}(x),\beta_{n}\alpha(x)'\in A_{+}$,
    	then $x'\in A_{+}$.
    	
    	For $x,~y\in F_{i}$ with $xy=0$, then $\overline{\varphi}(x)\overline{\varphi}(y)=\beta_{n}\varphi'(x)\beta_{n}\varphi'(y).$
    	By $(3')$,  for $\delta>0$, there exists sufficiently large  $N>0$, such that $n>N$, one has  $\|\beta_{n}\varphi'(x)\beta_{n}\varphi'(y)-\beta_{n}(\varphi'(x)\varphi'(y))\|<\delta$. Since $\varphi'|_{F_{i}}$ is completely positive contractive order zero map, then $\overline{\varphi}|_{F_{i}}$ is $\delta$-almost order zero map. By Theorem \ref{thm:2.19},
    	there exists $\varphi:F\rightarrow A$ satisfying $\varphi|_{F_{i}}$ is a completely positive contractive order zero map and
    	$\|\varphi-\overline{\varphi}\|\leq\varepsilon$.
    	
    	Then for all $x\in G$, we have
    	
    	$\|x-x'-\varphi\psi(x)\|=\|x-\gamma_{n}(x)-(\beta_{n}\alpha(x)'-2\varepsilon)_{+}-\varphi\psi(x)$
    	
    	$\leq2\varepsilon+\|x-\gamma_{n}(x)-\beta_{n}\alpha(x)'-\overline{\varphi}\psi(x)\|+
    	\|\overline{\varphi}\psi(x)-\varphi\psi(x)\|$
    	
    	$\leq3\varepsilon+\|x-\gamma_{n}(x)-\beta_{n}\varphi'\psi'\alpha(x)-\beta_{n}\alpha(x)'\|$
    	
    	$\leq3\varepsilon+\|x-\gamma_{n}(x)-\beta_{n}\alpha(x)\|+
    	\|\beta_{n}\alpha(x)-\beta_{n}\varphi'\psi'\alpha(x)-\beta_{n}\alpha(x)'\|$
    	
    	$\leq3\varepsilon+\varepsilon'+	\|\alpha(x)-\varphi'\psi'\alpha(x)-\alpha(x)'\|$
    	
    	$\leq3\varepsilon+\varepsilon'+\varepsilon'$
    	
    	$\leq5\varepsilon$.
    	
    Since $\alpha(x)'\precsim \alpha(a_2)$, then  there exists $v\in A$ such that
    $$\|v\alpha(a_2)v^*-\alpha(x)'\|<\varepsilon'.$$  By $(3')$, we have $$\|\beta_n(v\alpha(a_2)v^*)-\beta_n\alpha(x)'\|<\varepsilon',$$  by $(3')$,
     $$\|\beta_n(v)\beta_n\alpha(a_2)\beta_n(v^*)-\beta_n\alpha(x)'\|<2\varepsilon'.$$ Therefore, we have
     $$(\beta_{n}\alpha(x)'-2\varepsilon)_{+}\precsim (\beta_{n}\alpha(a_{2})-\varepsilon')_{+}.$$
Since $\|a_2-\gamma_n(a_2)-\beta_n\alpha(a_2)\|<\varepsilon'$, we have $$(\beta_n\alpha(a_2)-\varepsilon')_++\gamma_n(a_2)\precsim a_2.$$
    Then, we have
    $$(\beta_n\alpha(a_2)-\varepsilon')_+\precsim a_2.$$

    Therefore, we  have

    $x'=\gamma_{n}(x)+(\beta_{n}\alpha(x)'-2\varepsilon)_{+}$

    $\precsim\gamma_{n}(1_{A})\oplus
    	(\beta_{n}\alpha(a_{2})-\varepsilon')_{+}\precsim\gamma_{n}(1_{A})\oplus a_{2}$

    $\precsim a_{1}+a_{2}\precsim a$.
    	
    	Since $\psi=\psi'\alpha$ and  $\|\psi'\|\leq 1$, then  $\alpha$ is a completely positive contractive map.
    	
    	Therefore, ${\rm T^2dim_{nuc}}(A)\leq n$.
    \end{proof}

    \begin{corollary} \label{cor:3.6} Let $A$ be an infinite-dimensional  simple separable amenable unital
 ${\rm C^*}$-algebra with ${\rm T^2dim_{nuc}}(A)\leq n$.
Suppose that  $\alpha:G\to {\rm Aut}(A)$ is  an action of a finite group $G$ on $A$
 which has the  tracial Rokhlin property.  Then ${\rm T^2dim_{nuc}}{{\rm C^*}(G, A,\alpha)}\leq n$.
\end{corollary}
\begin{proof}
By Theorem \ref{thm:3.1}, Theorem \ref{thm:3.2} and Theorem \ref{thm:3.5}.

\end{proof}

\end{document}